\documentclass[5p,preprint]{elsarticle}

\usepackage
{
    graphicx,
    amssymb,
    amsmath,
    amsthm,
    dsfont, 
    xcolor,
    mathtools,
    nicefrac,
    array,
    enumerate,
    color
}

\usepackage[notrig]{physics}
\usepackage[utf8]{inputenc}
\usepackage{todonotes}

\usepackage[pdffitwindow=false,
            plainpages=false,
            pdfpagelabels=true,
            pdfpagemode=UseOutlines,
            pdfpagelayout=SinglePage,
            bookmarks=false,
            colorlinks=true,
            hyperfootnotes=false,
            linkcolor=blue,
            urlcolor=blue!30!black,
            citecolor=green!50!black]{hyperref}

\usepackage[bf,normal]{caption}

\usepackage{natbib}        
\usepackage[svgnames]{xcolor}
\usepackage[shortlabels]{enumitem}
\usepackage{float}
\usepackage{caption}
\usepackage{subcaption}
\usepackage{tikz}
\usepackage[expansion=false]{microtype}
\usepackage{url}
\usepackage{comment}

\newcommand{\nl}{\left\|}
\newcommand{\nr}{\right\|}

\newcommand{\Norm}[2][ ]{\nl #2 \nr_{#1}}

\newcommand{\N}{\mathds{N}}

\newcommand{\R}{\mathds{R}}
\newcommand{\Rp}{\R_{\geq0}}

\renewcommand{\phi}{\varphi}

\newcommand{\cC}{\mathcal{C}}

\newcommand{\cE}{\mathcal{E}}
\newcommand{\cF}{\mathcal{F}}

\newcommand{\oT}{\mathbf{T}}

\newcommand{\Impl}{\Longrightarrow}

\newcommand{\fa}{\forall \, }
\newcommand{\ex}{\exists \, }

\makeatletter
\newcommand{\Itemlabel}[2]{#2\def\@currentlabel{#2}\label{#1}}
\makeatother

\makeatletter
\newcommand\setcurrentname[1]{\def\@currentlabelname{#1}}
\makeatother

\DeclareMathOperator*{\esssup}{ess\,sup}

\DeclareMathOperator*{\loc}{loc}

\newcommand{\nocontentsline}[3]{}
\newcommand{\tocless}[2]{\bgroup\let\addcontentsline=\nocontentsline#1{#2}\egroup}
\newcommand{\toclesslab}[3]{\bgroup\let\addcontentsline=\nocontentsline#1{#2\label{#3}}\egroup} 

\newtheorem{thm}{Theorem}

\newtheorem{remark}[thm]{Remark}

\newcommand{\ddt}{\tfrac{\text{\normalfont d}}{\text{\normalfont d}t}}
\allowdisplaybreaks

\begin{document}

\begin{frontmatter}

\title{Funnel control with input filter for nonlinear systems of higher relative degree\tnoteref{funding}}

\tnotetext[funding]{This work was funded by the Deutsche Forschungsgemeinschaft (DFG, German Research Foundation)~-- Project-ID 524064985.}
\author[a]{Janina Schaa}\ead{janina.schaa@mathematik.uni-halle.de}
\author[a]{Thomas Berger}\ead{thomas.berger@mathematik.uni-halle.de}
\address[a]{Institut f\"ur Mathematik, Martin-Luther-Universit\"at Halle-Wittenberg, Halle (Saale), Germany}


\begin{abstract}
This paper addresses output reference tracking with prescribed transient performance for unknown nonlinear multi-input multi-output systems of arbitrary relative degree. We propose a novel derivative-free extension of funnel control based on a set of auxiliary filter variables that estimate the output derivatives. The resulting controller ensures that the tracking error evolves within prescribed performance bounds, while avoiding differentiation of the output signal and maintaining a simple structure with only a small number of tuning parameters. The effectiveness of the proposed approach is validated by a numerical example.
\end{abstract}

\end{frontmatter}

\section{Introduction} \label{sec:introduction}
In many control applications, achieving output reference tracking with prescribed transient performance for uncertain nonlinear systems is a challenging task. Funnel control, introduced in~\cite{Ilchmann2002}, provides an adaptive control framework to address this problem for a broad class of nonlinear multi-input multi-output systems via high-gain feedback, see, e.g.,~\cite{SurveyBerger}. By enforcing the tracking error to evolve within a chosen performance funnel, funnel controllers ensure transient behavior under comparatively mild structural assumptions on the system.

Due to their simple structure, funnel controllers have been shown to be an appropriate choice for many real-world applications, including DC-link power flow control~\cite{Senfelds2014}, speed control of wind turbine systems~\cite{Hackl2015Wind, Hackl2017}, control of peak inspiratory pressure~\cite{Pomprapa2015}, control of industrial servo-systems~\cite{Ilchmann2009, Hackl2013, Hackl2017} and voltage and current control of electrical circuits~\cite{Berger2014}.

One approach closely related to funnel control is prescribed performance control (PPC), see \cite{Bechlioulis2008, BechRovi14}. Both aim to guarantee prescribed transient behavior by constraining the tracking error within suitably defined bounds. The two approaches, however, differ in their controller design. In conventional PPC, the prescribed performance requirements are typically imposed through a performance function and a nonlinear transformation of the tracking error, after which a controller is designed for the resulting transformed system \cite{bu_ppc_2023}. Funnel control follows a different design philosophy: rather than transforming the constrained tracking error, the performance requirement is incorporated directly into the feedback law through a funnel function chosen from a large class of admissible functions. Furthermore, both approaches are typically designed for different system classes. For more details on their comparison see also the survey~\cite{SurveyBerger}.

Despite these advantages of funnel control, several challenges concerning the practical application of funnel controllers have been investigated in recent research, such as the treatment of input constraints~\cite{Hu2022, Berger2024}, discrete-time systems~\cite{Cheng2023} and output measurement losses~\cite{Berger2023}. A persistent issue, however, is the application of funnel controllers to systems with higher relative degree. Classical funnel control schemes like the approach discussed in~\cite{Berger2018} require the availability of output derivatives, which are generally not directly accessible in practice. This limits the practical implementation of such controllers, particularly in applications where only the system output is measured.

This issue has been discussed in several works in the literature. For single-input single-output systems, solutions have been proposed in~\cite{Miller1991} for the linear time-invariant case, and in~\cite{ChowdhuryKhalil2017} by the introduction of a virtual output, whose parameters depend on the system dynamics, thereby making the design model-dependent. A filter-based extension has been considered in~\cite{Ilchmann2007}, resulting in a complex controller structure that involves high powers of a potentially large gain function. In the PPC framework, a model-free output feedback approach for nonlinear systems without access to the derivatives of the output has been proposed in~\cite{Dimanidis2020}.

Another approach is based on the introduction of a pre-compensator, proposed in~\cite{BergerReis18} (and with a proof for the general case given in~\cite{Lanz22}), in order to estimate the required derivative signals. However, this approach results in high complexity due to a large number of auxiliary variables, in particular for systems of higher relative degree. Furthermore, derivative-free funnel control based on sampled data has been investigated for systems of relative degree two, see~\cite{Lanza2024}. However, this approach requires an appropriate choice of the sampling rate, which depends on system-specific properties.

In this paper, we address this issue by proposing a novel extension of funnel control utilizing a collection of filter signals that approximate the output derivatives, thereby achieving output tracking with prescribed performance for nonlinear systems of arbitrary relative degree without requiring access to the output derivatives or numerical differentiation of the output signal. The proposed approach extends the results of~\cite{Dennstaedt2025} from relative degree two to arbitrary relative degree. While the proposed filter structure has a recursive form, this extension is not obtained by a straightforward repetition of the construction from~\cite{Dennstaedt2025}; rather, we develop a suitable technique to handle the general case here. In contrast to the derivative-free design based on a cascade of pre-compensators introduced in~\cite{BergerReis18}, the proposed approach requires fewer auxiliary dynamics, resulting in a lower-dimensional closed-loop system with a comparatively simpler controller structure and fewer tuning parameters. Overall, the proposed method retains the model-free nature of funnel control while providing a derivative-free implementation for systems with arbitrary relative degree.

The remainder of this paper is organized as follows. In Section~2, we introduce the system class and the control objective. Section~3 presents the controller design and establishes the main result. A simulation example is discussed in Section~4, followed by concluding remarks in Section~5. 

\subsection{Notation}
$\N$ and $\R$ denote the natural and real numbers, respectively, $\N_0 \coloneq \N \cup \{ 0 \}$ and $\R_{\geq 0} \coloneq [0,\infty)$. $\Norm{x}$ denotes the Euclidean norm of $x\in \R^n$. $\cC^k(U,\R^n)$ is the linear space of $k$-times continuously differentiable functions $f:U \rightarrow \R^n$, where $k \in \N_0 \cup \{ \infty \}$ and $U \subseteq \R^m$. $L^{\infty}(I,\R^n)$ denotes the space of measurable and essentially bounded functions $f:I\rightarrow \R^n$ for some interval $I \subseteq \R$, equipped with the norm $\Vert f \Vert_{\infty} \coloneq \operatorname{ess} \sup_{t\in I} \Norm{f(t)}$. The set of measurable and locally essentially bounded functions is denoted by $L^{\infty}_{\text{loc}}(I,\R^n)$. The Sobolev space of all $k$-times weakly differentiable functions $f:I\rightarrow \R^n$ such that $f,\ldots,f^{(k)} \in L^{\infty}(I,\R^n)$, for some $k\in \N$, is $W^{k,\infty}(I,\R^n)$.

\section{Problem formulation} \label{sec:ProbForm}

\subsection{System class}

Consider a multi-input multi-output system governed by the differential equation
\begin{equation} \label{eq:System}
\begin{aligned}
y^{(r)}(t) &= \sum_{i = 1}^{r-1}R_i y^{(i)}(t) + f\big(\oT(y,\ldots,y^{(r-1)})(t) \big) + \Gamma u(t), \\
y|_{[0,t_0]} &= y^0  \in \cC^{r-1}([0,t_0],\R^n),
\end{aligned}
\end{equation}
of order $r \geq 2$, where $t_0 \geq 0$, $R_i \in \R^{n \times n}$ are matrices, $\Gamma \in \R^{n\times n}$ has a positive definite symmetric part $\frac{1}{2}(\Gamma + \Gamma^{\top})$, $u\in L_{\text{loc}}^{\infty}(\R_{\geq 0},\R^n)$ is a control input, and $y(t)\in \R^n$ at time $t \geq 0$. Further, $y^0\in \cC^{r-1}([0,t_0],\R^n)$ is an initial trajectory, $f\in\cC(\R^q,\R^n)$ is a nonlinear function, and ${\oT:\cC(\R_{\geq 0},\R^n)^{r}\to L^\infty_{\loc}([t_0,\infty),\R^q)}$ is a nonlinear operator satisfying the following conditions:
\begin{enumerate}[(i)]
 \item\emph{Causality}:  $\fa y_1,y_2\in\cC(\R_{\geq 0},\R^n)^r$  $\fa t\geq t_0$:
    \[
        y_1\vert_{[0,t]} = y_2\vert_{[0,t]}
        \quad \Impl\quad
        \oT(y_1)\vert_{[t_0,t]}=\oT(y_2)\vert_{[t_0,t]}.
    \]
    \item\emph{Local Lipschitz}: 
    $\fa t \ge t_0 $ $\fa y \in \cC([0,t], \R^n)^r$ 
    ${\ex \Delta, \delta, c > 0}$ 
    $\fa y_1, y_2 \in \cC(\Rp, \R^n)^r$ with
    $y_1|_{[0,t]} = y_2|_{[0,t]} = y $ 
    and $\Norm{y_1(s) - y(t)} < \delta$,  $\Norm{y_2(s) - y(t)} < \delta $ for all $s \in [t,t+\Delta]$:
    \[
     \hspace*{-2mm}   \esssup_{\mathclap{s \in [t,t+\Delta]}}  \Norm{\oT(y_1)(s) \!-\! \oT(y_2)(s) }  
        \!\le\! c \ \sup_{\mathclap{s \in [t,t+\Delta]}}\ \Norm{y_1(s)\!-\! y_2(s)}\!.
    \] 
    \item\label{Item:PropBIBO}\emph{Bounded-input bounded-output (BIBO)}:\\
    $\fa c_0 > 0$ $\ex c_1>0$  $\fa y_1,\ldots,y_r \in \cC(\Rp, \R^n)$:
    \begin{multline*}
    \sup_{t \in \R_{\geq 0}} \Norm{y_1(t)} \le c_0 \\
    \Impl \ \esssup_{t \in [t_0,\infty)} \Norm{\oT (y_1,\ldots,y_r)(t)}  \le c_1.
    \end{multline*}
\end{enumerate}
The linear terms $\sum_{i = 1}^{r-1}R_i y^{(i)}(t)$ appear explicitly in the system equation, since they are excluded by the BIBO property~(iii) of operator $\oT$. Encompassed by the operator formulation in~\eqref{eq:System} are linear systems and systems with fixed delays. However, measurement noise and discretization effects are not covered by this system class.
\begin{remark}
Observe that property~(iii) of the operator~$\oT$ ensures that a bounded system output $y$ is sufficient to conclude that the internal dynamics of the system, which are modeled by the operator $\oT$, remain bounded. This is a stronger assumption than the BIBO property used in classical funnel control, where the output derivatives are assumed to be available, see \cite{Berger2018, Berger2021}.
\end{remark}

\subsection{Control objective} \label{Ssec:ContrObj}

The objective is to design an output feedback control law such that, for any reference signal $y_{\text{ref}}\in W^{1,\infty}([t_0,\infty),\R^n)$, the output tracking error $e(t) \coloneq y(t) - y_{\text{ref}}(t)$ evolves within a prescribed performance funnel
\begin{equation} \label{Funnel}
\cF_{\varphi} \coloneq \left\{ (t,e) \in [t_0,\infty) \times \R^n \mid \varphi(t) \Norm{e} < 1  \right\},
\end{equation}
determined by the (user-defined) choice of the function $\varphi \in W^{1,\infty}([t_0,\infty), \R)$ such that $\inf_{t\geq t_0}\varphi(t) > 0$. 

\section{Controller design} \label{sec:ContDesign}
We choose the following output feedback control law
\begin{equation}\label{eq:Controller}
\boxed{
\begin{aligned}
	e(t) &=y(t)-y_{\text{ref}}(t),\\
	\dot{\xi}_{r-1}(t) &= -\xi_{r-1}(t) + u(t),	\\
	\dot{\xi}_{i}(t) &= -(r-i)\xi_i(t) + \xi_{i+1}(t), &\fa 1 \leq i < r-1,\\
	\theta_1(t) &= \xi_1(t) - \frac{-e(t)}{1-\varphi(t)^2\Norm{e(t)}^2},\\
	\theta_i(t) &= \xi_i(t) - \frac{-\theta_{i-1}(t)}{\hat{\theta}_{i-1}^2 - \Norm{\theta_{i-1}(t)}^2},& \fa 2 \leq i \le r-1,\\
	u(t) &= \frac{-\vartheta \theta_{r-1}(t)}{\hat{\theta}_{r-1}^2 - \Norm{\theta_{r-1}(t)}^2}, & \vartheta > 0
\end{aligned}
}    
\end{equation}
with constant initial conditions $ \xi_i(t) = \xi_i^0$ for all $0\leq t\leq t_0$ and parameters $\vartheta > 0$ and $ \hat{\theta}_{i} > 0$ for all $1 \leq i \leq r-1$.

Observe that this controller does not use the derivatives $\dot{y}, \ldots, y^{(r-1)}$ of the output, nor the derivatives $\dot{y}_{\text{ref}}, \ldots, y^{(r-1)}_{\text{ref}}$ of the reference trajectory. Instead, we introduce $r-1$ filter variables $\xi_1, \ldots, \xi_{r-1}$, which qualitatively reproduce the signals $\dot{y}, \ldots, y^{(r-1)} $, respectively. The control law $u$ is then designed to ensure that the filter variable $\xi_1$ tracks the reference signal $ \frac{-e(t)}{1-\varphi(t)^2\Norm{e(t)}^2}$, which corresponds to the funnel control law for systems of relative degree one. The tracking is achieved up to an error $\theta_1$, which remains bounded by the constant $\hat{\theta}_1$, due to the iterative construction of constant funnel bounds for the error signals $\theta_i$. The controller has $r$ design variables, $\hat{\theta}_1, \ldots, \hat{\theta}_{r-1}$ and $\vartheta$. Moreover, the initial values $\xi_i^0$ must be specified. All parameters can be selected without knowledge of the system, except for the initial system output $y^0(t_0)$. 

Although the proposed filter structure has a recursive form, its extension from relative degree two to arbitrary relative degree is not a direct recursive repetition of the analysis in~\cite{Dennstaedt2025}. In particular, the filter variables are coupled, which results in coupled estimation errors whose boundedness cannot, in general, be established individually by applying the argument from~\cite{Dennstaedt2025} recursively. The analysis therefore relies on carefully constructed linear combinations of the estimation errors: first, boundedness of these combinations is established, and subsequently the desired bounds on the individual estimation errors are recovered from them. This error analysis constitutes the main technical difficulty in extending the construction of the derivative-free funnel controller to systems of arbitrary relative degree.

A function $(x,\xi):[0,\omega) \rightarrow \R^{rn}\times \R^{(r-1)n}$, $\omega \in (0,\infty],$ is called a solution of the closed-loop system~\eqref{eq:System},~\eqref{eq:Controller}, if it satisfies the initial conditions $(x_1(t),\xi_1(t),\ldots,\xi_{r-1}(t)) = (y^0(t), \xi_1^0,\ldots,\xi_{r-1}^0)$ for each $t \in [0,t_0],$ and $(x,\xi)_{\vert [t_0,\omega)}$ is absolutely continuous and satisfies
\begin{align*}
\dot{x}_i(t) &= x_{i+1}(t), \quad \fa 1 \leq i \leq r-1\\
\dot{x}_r(t) &=  \sum_{i=1}^{r-1} R_i x_{i+1}(t) + f\big(\oT(x_1,\ldots, x_{r})(t) \big) + \Gamma u(t)\\
\dot{\xi}_i(t) &= -(r-i)\xi_i(t) + \xi_{i+1}(t), \quad \fa 1 \leq i < r-1\\ 
\dot{\xi}_{r-1}(t) &= -\xi_{r-1}(t) + u(t), 
\end{align*}
with $u$ as defined in~\eqref{eq:Controller} for almost all $t \in [t_0,\omega)$. We call a solution maximal if it does not have a right extension, which is also a solution.

\begin{thm}
Consider the system~\eqref{eq:System}, choose a function $\varphi \in W^{1,\infty}([t_0,\infty), \R)$ such that $\inf_{t\geq t_0}\varphi(t) > 0$, and a reference trajectory $y_{\rm{ref}}\in W^{1,\infty}([t_0,\infty),\R^n)$. Choose parameters $\hat{\theta}_i > 0$ for each $1 \leq i \leq r-1$, and initial values $y^0 \in\cC^{r-1}([0,t_0],\R^n) $ and $ \xi_i^0 \in \R^n$ for each $1 \leq i \leq r-1$, such that the following conditions are satisfied.
\begin{itemize}
\item $\varphi(t_0) \Norm{y^0(t_0) - y_{\rm{ref}}(t_0)} < 1$,
\item $\theta_{1}^0\! \coloneq \!\xi_1^0 \!-\! \frac{-(y^0(t_0) - y_{\rm{ref}}(t_0))}{1-\varphi(t_0)^2 \Norm{y^0(t_0) - y_{\rm{ref}}(t_0)}^2}$ satisfies $\left\Vert \theta_1^0  \right\Vert\! <\! \hat{\theta}_1 $,
\item $\theta_i^0 \coloneq \xi_{i}^0 - \frac{-\theta_{i-1}^0}{\hat{\theta}_{i-1}^2 - \Norm{\theta_{i-1}^0}^2}$ satisfies $\Norm{\theta_i^0} < \hat{\theta}_{i}$ for each $2 \leq i \leq r-1$.
\end{itemize}
Then, the application of the controller~\eqref{eq:Controller} to the system~\eqref{eq:System} yields an initial-value problem which has a solution, and every maximal solution $(x,\xi):[0,\omega) \rightarrow \R^{rn} \times \R^{(r-1)n}$, $\omega \in (0,\infty]$, has the following properties:
\begin{itemize}
\item The solution is global, i.e., $\omega = \infty$.
\item The tracking error $e = x_1 - y_{\rm{ref}}$ evolves uniformly within the performance funnel $\cF_{\varphi},$ i.e.,
$$ \exists\, \varepsilon \in (0,1) \ \forall\, t \geq t_0 : \quad \varphi(t) \Norm{e(t)} < \varepsilon. $$
\item The input signal and filter variables $u, \xi_i :[t_0,\infty) \rightarrow \R^n$ are bounded for $1\le i\le r-1$.
\end{itemize}
\end{thm}
\begin{proof}
The proof is divided into the following steps. First, we show the existence of a maximal solution of the closed-loop system. Then, we show boundedness of the estimation error $\dot{y}-\Gamma \xi_1$. Subsequently, we will establish the boundedness of a family of signals including the higher-order estimation errors $y^{(j)} - \Gamma \xi_j$. Next, we will show that the error remains in a compact subset of the interior of the performance funnel. An inductive argument is used to conclude boundedness of the closed-loop signals $u$, $\xi_j$ and $y^{(j)}$. Finally, these properties are combined to establish the global existence of the closed-loop solution.

\emph{Step 1:} We show that there exists a maximal solution $(x,\xi):[0,\omega) \rightarrow \R^{rn} \times \R^{(r-1)n}$, $\omega \in (0,\infty]$, of the closed-loop system~\eqref{eq:System},\eqref{eq:Controller}. Define the open set on which the closed-loop system is well-defined,
\begin{align*}
&\cE  \coloneq \\ &\left \{ \!(t,y,z) \!\in \![t_0,\!\infty)\! \times \!\R^{rn}\! \times\! \R^{(r-1)n} \middle\vert\! \!\begin{array}{l}  \varphi(t)\! \Norm{y_1 - y_{\text{ref}}(t)} \!<\! 1, \!\!\!\! \\  \left\Vert \theta^i \right\Vert < \hat{\theta}_{i}, \\ \fa 1 \leq i \leq r-1  \\  \end{array}  \right\}
\end{align*}
using the notation $\theta^1 \coloneq z_1 - \frac{-(y_1 - y_{\text{ref}}(t))}{1-\varphi(t)^2\Norm{ y_1 - y_{\text{ref}}(t) }},$ and $\theta^{i} \coloneq z_{i} - \frac{-\theta^{i-1}}{\hat{\theta}_{i-1}^2 - \Norm{\theta^{i-1}}^2} $ for each $1 < i \leq r-1$.  We can now describe the closed-loop dynamics on the set $\cE$ by the following construction.
Define the set of functions $g_1(t,y_1,z_1) = z_1 - \frac{-(y_1 - y_{\text{ref}}(t))}{1-\varphi(t)^2 \Norm{y_1 - y_{\text{ref}}(t)}^2}$ and similarly for each $2\leq i \leq r-1$, $g_i(t,y_1,z_1,\ldots,z_i) = z_i - \frac{-g_{i-1}(t,y_1,z_1,\ldots,z_{i-1})}{\hat{\theta}_{i-1}^2 - \Norm{g_{i-1}(t,y_1,z_1,\ldots, z_{i-1})}^2}$, and $g_r(t,y_1,z_1,\ldots,z_{r-1}) = \frac{-\vartheta g_{r-1}(t,y_1,z_1,\ldots,z_{r-1})}{\hat{\theta}_{r-1}^2 - \Norm{g_{r-1}(t,y_1,z_1,\ldots,z_{r-1})}^2}$. Further, introduce the function 
\[
  G(t,y,z) = \sum_{i=1}^{r-1}R_i y_{i+1}  +\Gamma g_r(t,y_1,z_1,\ldots,z_{r-1})
\]
in order to define the function $F:\cE \times \R^q\rightarrow \R^{(2r-1)n}$ by
\begin{equation*}
F(t,y,z,\eta) = \begin{pmatrix} y_2 \\ \vdots \\ y_r \\ G(t,y,z) + f(\eta) \\ -(r-1)z_1 + z_2 \\ -(r-2)z_2+z_3 \\ \vdots \\ -2z_{r-2}+z_{r-1} \\ -z_{r-1} + g_r(t,y_1,z_1,\ldots,z_{r-1}) \end{pmatrix}.
\end{equation*}
 Then the initial value problem~\eqref{eq:System},\eqref{eq:Controller} takes the form 
\begin{align*}
\begin{pmatrix} \dot{x}(t) \\ \dot{\xi}(t) \end{pmatrix} &= F(t,x(t),\xi(t),\oT(x)(t)), \\ 
(x,\xi)|_{[0,t_0]} &= (x^0,\xi^0),
\end{align*}
where, for $t\in [0,t_0]$,
\[
  x^0(t)\coloneq \begin{pmatrix}  y^0(t) \\ \vdots\\ (y^0)^{(r-1)}(t)\end{pmatrix},\quad \xi^0(t) \coloneq \begin{pmatrix} \xi_1^0 \\ \vdots \\ \xi_{r-1}^0 \end{pmatrix}.
 \]
By assumption, the initial values satisfy 
$$\big(t_0, x^0(t_0),\xi^0(t_0) \big) \in \cE.$$ 
The function $F$ is measurable in~$t$, continuous in~$(y,z,\eta)$ and locally essentially bounded. Therefore, an application of a variant of~\cite[Thm.~B.1]{IlchRyan09} yields the existence of a solution of~\eqref{eq:System},\eqref{eq:Controller} and every solution can be extended to a maximal solution. Furthermore, any maximal solution $(x,\xi):[0,\omega) \rightarrow \R^{rn} \times \R^{(r-1)n}$, $\omega\in(0,\infty]$, has the property that its graph is not a compact subset of $\cE$, which will be essential to conclude the existence of the global solution in the last step.

\emph{Step 2:} We show boundedness of the estimation error $\dot{y} - \Gamma \xi_1$ on $[t_0,\omega)$. For this, we will first recursively construct a family of signals $\zeta_i$ given by linear combinations of the state variables, which satisfy stable differential equations. Then, we will express the estimation error $\dot{y} - \Gamma \xi_1$ as a suitable linear combination of bounded terms.
For each $1 \leq i \leq r-1$ and $r-i+1 \leq j \leq r-1$, define 
\begin{equation} \label{eq:def-aij}
a_{i,j} \coloneq \frac{(i-1)!}{(i-(r-j))!} = \prod_{k=1}^{r-1-j} (i-k) > 0,
\end{equation}
and observe that these constants satisfy 
\begin{equation} \label{eq:aij}
a_{i,j-1} + a_{i,j} (r-j) = i a_{i,j}
\end{equation}
for all $1 \leq i \leq r-1$ and $r-i+2 \leq j \leq r-1$.

Next, we recursively define a set of matrix-valued polynomials. Let $S_{r-2}(s) = -R_{r-1}\in \R[s]^{n\times n}$ be constant, and for each $1 \leq j \leq r-2$, define
\begin{equation} \label{def:S-indu}
S_{j-1}(s) = -R_j - s S_{j}(s)\in \R[s]^{n\times n}.
\end{equation}
Obviously, the polynomial $S_j(s)$ is of degree $r-2-j$, for each $0 \leq j \leq r-2$, and does not depend on time, since the matrices $R_i$ are constant. Denote $S_{r-1}(s) = 0 \in \R[s]^{n\times n}$. We further define a set of time-dependent matrices, using the notation
$$ S^{i}(t) \coloneq \sum_{j = 0}^{r-2} S_j(i) y^{(j)}(t) $$
for each $1 \leq i \leq r-1$ and $t\in [t_0,\omega)$, which satisfies the following differential equation
\begin{align*}
\dot{S}^{i}(t) + i S^{i}(t) &= \sum_{j = 0}^{r-2} S_j(i) y^{(j+1)}(t) + i \sum_{j = 0}^{r-2} S_j(i) y^{(j)}(t)\\
&= \sum_{j = 1}^{r-2} \left( S_{j-1}(i) +iS_j(i) \right) y^{(j)}(t)  \\
&\quad\quad+  S_{r-2}(i) y^{(r-1)}(t) + iS_0(i) y(t)\\
&\overset{\eqref{def:S-indu}}{=} \sum_{j = 1}^{r-2} -R_j y^{(j)}(t) - R_{r-1} y^{(r-1)}(t) +iS_0(i)y(t)\\
&= -\sum_{j = 1}^{r-1} R_j y^{(j)}(t) + iS_0(i)y(t).
\end{align*}
Hence, 
\begin{equation} \label{eq:S-i-property}
-iS^{i}(t) = \dot{S}^{i}(t) + \sum_{j=1}^{r-1} R_jy^{(j)}(t) - i S_{0}(i) y(t)
\end{equation}
holds for each $1 \leq i \leq r-1$ and $t\in [t_0,\omega)$. We define
\begin{equation} \label{def:zeta-i}
\begin{aligned}
\zeta_i(t) &\coloneq \sum_{j=r-i}^{r-1} (-1)^{r-1-j}\left(  i^{r-1-j} y^{(j)}(t) - a_{i,j} \Gamma \xi_j(t) \right) \\
&\quad+ S^{i}(t) + \sum_{j=i+1}^{r} (-i)^{j-1} y^{(r-j)}(t)
\end{aligned}
\end{equation}
for all $t\in[t_0,\omega)$ and each $1 \leq i \leq r-1$. Further, define the polynomials
\begin{align*}
\alpha_j(s) &= (-s)^{r-1-j} I_n + S_{j}(s),\\
\beta_j(s) &= (-1)^{r-j} \prod_{k=1}^{r-1-j} (s-k), 
\end{align*}
for $1 \leq j \leq r-1$, which are of degree $r-1-j$. Observe that we can represent each $\zeta_i(t)$ as
\begin{equation} \label{eq:zeta-i-poly}
\begin{aligned}
&\zeta_i(t) = \sum_{j=r-i}^{r-1} (-1)^{r-j} a_{i,j}\Gamma \xi_j(t) + \sum_{j=r-i}^{r-1} (-i)^{r-1-j} y^{(j)}(t)\\
&\qquad  + S^{i}(t) + \sum_{j=0}^{r-1-i} (-i)^{r-1-j} y^{(j)}(t)\\
&\overset{\eqref{eq:def-aij}}{=}  \sum_{j=r-i}^{r-1}\!\! \beta_j(i) \Gamma \xi_j(t)\! +\! \sum_{j=0}^{r-1} (-i)^{r-1-j} y^{(j)}(t) \!+\! \sum_{j=0}^{r-1} S_j(i)y^{(j)}(t) \\
&= \sum_{j = 0}^{r-1} \alpha_{j}(i) y^{(j)}(t) + \sum_{j= 1}^{r-1} \beta_j(i) \Gamma \xi_j(t).
\end{aligned}
\end{equation}

\emph{Step 2a:} We show that $\zeta_i$ remains bounded for each $1\leq i\leq r-1$ by utilizing a Gronwall argument. Observe that the following differential equation holds for each $1 \leq i \leq r-1$ and $t\in [t_0,\omega)$.
\begin{align}
&\dot{\zeta}_i(t) = \sum_{j = r-i}^{r-1} (-1)^{r-1-j}(i^{r-1-j} y^{(j+1)}(t) - a_{i,j} \Gamma \dot{\xi}_j(t)) \notag\\
&\qquad+ \dot{S}^{i}(t) + \sum_{j=i+1}^{r} (-i)^{j-1} y^{(r-j+1)}(t) \notag\\
&\overset{\eqref{eq:S-i-property}}{=}  y^{(r)}(t) + \sum_{j = r-i}^{r-2} (-1)^{r-1-j}(i^{r-1-j} y^{(j+1)}(t) - a_{i,j} \Gamma \dot{\xi}_j(t))\notag \\
&\quad -iS^{i}(t) - \sum_{j=1}^{r-1} R_jy^{(j)}(t) + \sum_{j=i+1}^{r} (-i)^{j-1} y^{(r-j+1)}(t)\notag\\
&\quad  - a_{i,r-1} \Gamma (-\xi_{r-1}(t) + u(t))+ iS_{0}(i) y(t) \notag\\
&\overset{\eqref{eq:System},\eqref{eq:Controller}}{=} \! f(\oT(y,\ldots,y^{(r-1)})(t)) + \sum_{j=1}^{r-1} R_jy^{(j)}(t) + \Gamma u(t) \notag\\
&\quad  + \!\!\sum_{j=i+1}^{r} \!(-i)^{j-1} y^{(r-j+1)}(t) -iS^{i}(t)  + iS_{0}(i) y(t)\notag\\
&\quad  \!- \sum_{j=1}^{r-1} R_jy^{(j)}(t)+ \!\!\! \sum_{j = r-i}^{r-2}\!\! (-1)^{r-1-j}\!\Big(i^{r-1-j} y^{(j+1)}(t)\notag\\
&\quad -\! a_{i,j} \Gamma (-(r-j)\xi_j(t) + \xi_{j+1}(t))\Big) -\Gamma u(t) +\Gamma \xi_{r-1}(t) \notag\\
&=  f(\oT(y,\ldots,y^{(r-1)})(t)) -iS^{i}(t) + iS_{0}(i) y(t) \notag\\
&\quad + \sum_{j = r-i}^{r-2} (-1)^{r-1-j}a_{i,j} \Gamma (r-j)\xi_j(t) \notag\\
&\quad + \sum_{j=i+1}^{r} (-i)^{j-1} y^{(r-j+1)}(t)+\Gamma \xi_{r-1}(t)  \notag\\
&\quad + \sum_{j = r-i}^{r-2} (-1)^{r-1-j}(i^{r-1-j} y^{(j+1)}(t) - a_{i,j} \Gamma \xi_{j+1}(t)) \notag\\
&=f(\oT(y,\ldots,y^{(r-1)})(t)) -iS^i(t)+ \sum_{j=i}^{r-1} (-i)^{j} y^{(r-j)}(t)\notag\\
&\quad + \sum_{j = r-i+1}^{r-1} (-1)^{r-j}(i^{r-j} y^{(j)}(t) - a_{i,j-1} \Gamma \xi_{j}(t)) \notag \\
&\quad+ \sum_{j = r-i}^{r-1} (-1)^{r-1-j}a_{i,j} \Gamma (r-j)\xi_j(t)  + iS_{0}(i) y(t)\notag\\
&=f(\oT(y,\ldots,y^{(r-1)})(t))  -iS^i(t) + iS_{0}(i) y(t)\notag\\
&\quad +\!\!\! \!\! \sum_{j = r-i+1}^{r-1}\!\!\!\! (-1)^{r-j}(i^{r-j} y^{(j)}(t)\! - \!(a_{i,j-1}\!+\!a_{i,j}(r-j) )\Gamma \xi_{j}(t)) \notag\\
&\quad+ (-1)^{i-1} i a_{i,r-i} \Gamma \xi_{r-i}(t)+ (-i) \sum_{j=i}^{r-1} (-i)^{j-1} y^{(r-j)}(t) \notag  \\
&\overset{\eqref{eq:aij}}{=} f(\oT(y,\ldots,y^{(r-1)})(t))  -iS^{i}(t) + iS_{0}(i) y(t)\notag\\
&\quad+ (-i) \sum_{j=i}^{r-1} (-i)^{j-1} y^{(r-j)}(t)+ (-1)^{i-1} i a_{i,r-i} \Gamma \xi_{r-i}(t)\notag \\
&\quad + \sum_{j = r-i+1}^{r-1} (-1)^{r-j}(i^{r-j} y^{(j)}(t) - i a_{i,j}\Gamma \xi_{j}(t)) \notag\\
&=  f(\oT(y,\ldots,y^{(r-1)})(t)) + iS_{0}(i) y(t)\notag \\ 
&\quad+ (-1)^{i-1} i a_{i,r-i} \Gamma \xi_{r-i}(t) \notag \\
&\quad + (-i) \Bigg(  S^{i}(t) +  \sum_{j=i}^{r-1} (-i)^{j-1} y^{(r-j)}(t)\notag\\
&\quad\quad+ \sum_{j = r-i+1}^{r-1} (-1)^{r-j-1}(i^{r-j-1} y^{(j)}(t) - a_{i,j}\Gamma \xi_{j}(t)) \Bigg)\notag\\
&\overset{\eqref{def:zeta-i}}{=}  f(\oT(y,\ldots,y^{(r-1)})(t))  + (-1)^{i-1} i a_{i,r-i} \Gamma \xi_{r-i}(t) \notag\\ 
&\quad + (-i) \Big(\zeta_i(t) -  (-1)^{i-1}(i^{i-1} y^{(r-i)}(t) - a_{i,r-i} \Gamma \xi_{r-i}(t)) \notag\\
&\quad- (-i)^{r-1} y(t) + (-i)^{i-1} y^{(r-i)}(t)\Big)+ iS_{0}(i) y(t) \notag\\
&=  f(\oT(y,\ldots,y^{(r-1)})(t)) + iS_{0}(i) y(t)  \notag \\
&\quad+ (-i) \zeta_i(t)  - (-i)^{r} y(t). \label{diff-eq-zeta}
\end{align}
Using the fact that $y$ is bounded along the maximal solution due to the boundedness of the tracking error and the reference signal, and by the BIBO property of the operator $\oT$ and continuity of $f$, we can conclude that $f(\oT(y,\ldots,y^{(r-1)})(t))$ is bounded along the trajectory of the maximal solution. Therefore, $\zeta_i$ remains bounded by a Gronwall argument. 

\emph{Step 2b:} We show that $\dot{y} - \Gamma \xi_1$ can be expressed as a linear combination of bounded terms. Define the matrix
\begin{equation} \label{def:Matrix}
M \coloneq \begin{pmatrix} 1 & 1 & \hdots & 1 \\ 1 & 2 & \hdots & r-1 \\ 1 & 2^2 & \hdots & (r-1)^2 \\ \vdots & \vdots & \ddots & \vdots \\ 1 & 2^{(r-3)} & \hdots & (r-1)^{(r-3)} \end{pmatrix} \in \R^{(r-2) \times (r-1)}.
\end{equation}
Recall that for any pairwise distinct $\alpha_1,\ldots,\alpha_r \in \R$, the Vandermonde matrix
\begin{equation}
V(\alpha_1,\hdots, \alpha_{r}) \coloneq \begin{pmatrix} 1 & \alpha_1 & \alpha_1^2 & \hdots & \alpha_1^{r-1} \\ 1 & \alpha_2 & \alpha_2^2 & \hdots & \alpha_2^{r-1} \\ \vdots & \vdots & \vdots & \ddots & \vdots \\ 1 & \alpha_{r} & \alpha_r^2 & \hdots & \alpha_r^{r-1} \\ \end{pmatrix}
\end{equation}
is invertible. From this, we can conclude that $\dim \operatorname{ker} M = 1$, and moreover, by choosing an arbitrary $0 \neq c = \begin{pmatrix} c_1 \\ \vdots \\ c_{r-1} \end{pmatrix} \in \operatorname{ker} M$, we obtain 
\begin{equation} \label{eq:c-ker}
\sum_{i=1}^{r-1} c_i i^k = 0, \quad \fa 0 \leq k \leq r-3  , \quad  \sum_{i=1}^{r-1} c_i i^{r-2} \neq 0.
\end{equation}
Now, define $Z(t) = \sum_{i = 1}^{r-1} c_i \zeta_i(t)$ for $t\in [t_0,\omega)$. Using the representation of $\zeta_i$ from~\eqref{eq:zeta-i-poly}, we can express this as 
\begin{equation}
\begin{aligned}
Z(t)\! &= \!\sum_{j = 0}^{r-1}\! \left( \sum_{i=1}^{r-1} c_i \alpha_j(i)  \right)\! y^{(j)}(t)\! + \!\sum_{j=1}^{r-1}\! \left( \sum_{i=1}^{r-1} c_i \beta_j(i)\! \right)\! \Gamma \xi_j(t)\\
&= \sum_{j = 0}^{r-1} A_j y^{(j)}(t) + \sum_{j = 1}^{r-1} B_j \Gamma \xi_j(t),
\end{aligned}
\end{equation}
where the matrices $A_j = \sum_{i=1}^{r-1} c_i \alpha_j(i) = 0 $ for each $j \geq 2$ due to property~\eqref{eq:c-ker} and the fact that the $\alpha_j(s)$ are polynomials of degree $r-1-j \leq r-3$. Similarly, we obtain time-independent coefficients $B_j = \sum_{i=1}^{r-1} c_i \beta_j(i) = 0$ for each $j \geq 2$, again by property~\eqref{eq:c-ker} and the fact that the $\beta_j(s)$ are polynomials of degree $r-1-j \leq r-3$.

However, calculating the remaining coefficients shows that
\begin{align*} 
A_1 &= \sum_{i = 1}^{r-1} c_i\alpha_1(i) = \sum_{i = 1}^{r-1} c_i ((-i)^{r-2} I + S_{1}(i)) \\
&= \sum_{i = 1}^{r-1} c_i (-i)^{r-2} I \neq 0 ,
\end{align*}
since the degree of $S_{1}(s)$ is not larger than $r-3$ and by~\eqref{eq:c-ker}, and similarly
$$ B_1 = \sum_{i = 1}^{r-1} c_i \beta_1(i) = \sum_{i=1}^{r-1} c_i (-1)^{r-1} i^{r-2} =:- q \neq 0.$$
This shows that
\begin{equation}
Z(t) = q(\dot{y}(t) - \Gamma \xi_1(t)) + A_0 y(t).
\end{equation}
Since $Z(t)$ is bounded on $[t_0,\omega)$ as a linear combination of $\zeta_1(t),\ldots,\zeta_{r-1}(t)$, which are bounded by Step~2a, we can conclude that $$ \dot{y}(t) - \Gamma \xi_1(t) = \frac{1}{q} (Z(t) - A_0 y(t))$$
remains bounded on $[t_0,\omega)$. 

\emph{Step 3:} Construct a family $\left(Z_k(t)\right)_{1 \leq k \leq r-1}$ of bounded functions on $[t_0,\omega)$. Define $Z_1(t) \coloneq Z(t)$ as in the previous step. Fix $2 \leq k \leq r-1$. Choose $c^{(k)} \in \R^{r-1}$ such that 
\begin{equation*}
c^{(k)} = \begin{pmatrix} c_1^{(k)} \\ \vdots \\ c_{r-1}^{(k)} \end{pmatrix} \in \operatorname{ker}  \begin{pmatrix} 1 & 1 & \hdots & 1 \\ 1 & 2 & \hdots & r-1 \\ 1 & 2^2 & \hdots & (r-1)^2 \\ \vdots & \vdots & \ddots & \vdots \\ 1 & 2^{(r-2-k)} & \hdots & (r-1)^{(r-2-k)} \end{pmatrix}, 
\end{equation*}
and $\sum_{i=1}^{r-1} c_i^{(k)} i^{r-1-k} =: q_k \neq 0.$ A vector $c^{(k)}$ with these properties exists by the invertibility of the Vandermonde matrix.
Now, define $Z_k(t) = \sum_{i = 1}^{r-1} c_i^{(k)} \zeta_i(t)$ for $t\in[t_0,\omega)$. This is bounded, since all $\zeta_i$ are bounded by Step~2. Again, we can express
\begin{equation}
\begin{aligned}
Z_k(t)\! &= \!\sum_{j = 0}^{r-1}\! \left( \sum_{i=1}^{r-1} c_i^{(k)} \alpha_j(i)  \right) \!y^{(j)}(t) \!+\! \sum_{j=1}^{r-1} \!\left( \sum_{i=1}^{r-1} c_i^{(k)} \beta_j(i)\!\! \right)\! \Gamma \xi_j(t)\\
&= \sum_{j = 0}^{r-1} A_j^{(k)} y^{(j)}(t) + \sum_{j = 1}^{r-1} B_j^{(k)} \Gamma \xi_j(t),
\end{aligned}
\end{equation}
where $A_j^{(k)} = \sum_{i=1}^{r-1} c_i^{(k)} \alpha_j(i) = 0 $ for each $j \geq k+1$ due to the choice of $c^{(k)}$ and the fact that the $\alpha_j(s)$ are polynomials of degree $r-1-j \leq r-2-k$. Similarly, we obtain $B_j^{(k)} = \sum_{i=1}^{r-1} c_i^{(k)} \beta_j(i) = 0$ for each $j \geq k+1$, again by the choice of $c^{(k)}$ and the fact that the $\beta_j(s)$ are polynomials of degree $r-1-j \leq r-2-k$.

However, calculating the coefficient for $j=k$ shows that 
\begin{align*}
A_k^{(k)} &= \sum_{i = 1}^{r-1} c_i^{(k)} \alpha_k(i) = \sum_{i = 1}^{r-1} c_i^{(k)} ((-i)^{r-1-k} I + S_{k}(i)) \\
&= \sum_{i = 1}^{r-1} c_i^{(k)} (-i)^{r-1-k} I = (-1)^{r-1-k} q_k I \neq 0 , 
\end{align*}
since the degree of $S_{k}(s)$ is not larger than $r-1-k$, and similarly
$$ B_k^{(k)}\!\! = \!\!\sum_{i = 1}^{r-1} c_i^{(k)} \!\beta_k(i)\! = \!\sum_{i=1}^{r-1}\! c_i^{(k)} (-1)^{r-k} i^{r-1-k}\!\! = \!(-1)^{r-k}\! q_k \!\neq\! 0.$$
This shows that 
\begin{equation} \label{eq:Z_k}
Z_k(t) = \sum_{j=0}^{k} A_j^{(k)} y^{(j)}(t) + \sum_{j=1}^{k} B_j^{(k)} \Gamma \xi_j(t)
\end{equation}
for all $t\in [t_0,\omega)$, for some matrix $A_k^{(k)} = (-1)^{r-1-k} q_k I$ that is invertible. 

\emph{Step 4:} We show that $\varphi(t)\Norm{e(t)} \leq \varepsilon$ for all $t \in [t_0,\omega)$, for some $\varepsilon \in (0,1).$ Choose $\varepsilon \in (0,1)$ such that $\varphi(t_0) \Norm{e(t_0)} < \varepsilon$, and 
\begin{equation} \label{def-eps}
\begin{aligned}
 \frac{\varepsilon^2}{1-\varepsilon^2} &> \frac{1}{\lambda_{\text{min}}(\frac{1}{2}(\Gamma + \Gamma^{\top}))} \Bigg( \Vert \dot{y}_{\text{ref}} \Vert_{\infty} \Vert \varphi \Vert_{\infty}  + \left\Vert \frac{\dot{\varphi}}{\varphi} \right\Vert_{\infty} \\
&\quad\quad+\Vert \varphi \Vert_{\infty} (\lambda_{\text{max}}(\Gamma)\hat{\theta}_1 + \Vert \dot{y} - \Gamma\xi_1 \Vert_{\infty}) \Bigg) =: \sigma,
\end{aligned}
\end{equation}
where $\lambda_{\text{min}}(\frac{1}{2}(\Gamma + \Gamma^{\top})) > 0$ is the smallest eigenvalue of the positive definite matrix $\frac{1}{2}(\Gamma + \Gamma^{\top})$, and $\lambda_{\text{max}}(\Gamma)>0$ is the absolute value of the largest eigenvalue of the matrix $\Gamma$, and $\sigma$ is finite due to Step~2 and the assumptions on $\varphi$ and $y_{\text{ref}}$. Seeking a contradiction, assume that there exists $s_1 \in [t_0,\omega)$ such that $\varphi(s_1) \Norm{e(s_1)} > \varepsilon$. Since $\varphi(t_0)\Norm{e(t_0)}<\varepsilon$ by choice of $\varepsilon$, there exists $s_0 \coloneq \sup \{ t \in [t_0,s_1] \mid \varphi(t)\Norm{e(t)} = \varepsilon \}$. 

We will use the fact that
\begin{equation} \label{eq:doty}
\begin{aligned}
\dot{y}(t) &=\Gamma\xi_1(t) + \dot{y}(t) -\Gamma \xi_1(t) \\
&= \Gamma\theta_1(t) + \Gamma\frac{-e(t)}{1-\varphi(t)^2\Norm{e(t)}^2} + \dot{y}(t) - \Gamma\xi_1(t)
\end{aligned}
\end{equation}
holds for all $t\in [t_0,\omega)$.
Observe that for all $s \in [s_0,s_1]$, the following estimate holds:

\begin{align*}
&\tfrac{1}{2} \frac{d}{ds} \varphi(s)^2\! \Norm{e(s)}^2 = \varphi(s)^2 \langle e(s), \dot{e}(s) \rangle + \varphi(s) \dot{\varphi}(s) \Norm{e(s)}^2 \\
&\qquad \leq  \varphi(s)^2 \langle e(s) , \dot{y}(s) \rangle \!+\! \Vert \dot{y}_{\text{ref}} \Vert_{\infty} \Vert \varphi \Vert_{\infty}  \!+\! \left\Vert \frac{\dot{\varphi}}{\varphi} \right\Vert_{\infty} \\
&\qquad\overset{\eqref{eq:doty}}{=}  \varphi(s)^2 \left\langle e(s), \frac{-\Gamma e(s)}{1-\varphi(s)^2\Norm{e(s)}^2} \right\rangle \\
&\qquad\quad+ \varphi(s)^2 \langle e(s), \Gamma\theta_1(s) + \dot{y}(s) - \Gamma\xi_1(s) \rangle \\
&\qquad\quad + \Vert \dot{y}_{\text{ref}} \Vert_{\infty} \Vert \varphi \Vert_{\infty}  + \left\Vert \frac{\dot{\varphi}}{\varphi} \right\Vert_{\infty} \\
&\qquad\leq -\frac{\lambda_{\text{min}}(\frac{1}{2}(\Gamma + \Gamma^{\top}))\varepsilon^2}{1-\varepsilon^2} +\sigma \overset{\eqref{def-eps}}{<}0.
\end{align*} 
This implies that $\varepsilon=\varphi(s_0)\Norm{e(s_0)} > \varphi(s_1)\Norm{e(s_1)} > \varepsilon$, a contradiction. Hence, the claim of Step~4 is shown. 

\emph{Step 5:} We inductively show boundedness of $u$, $\xi_i$ and $y^{(i)}$ for each $1 \leq i \leq r-1$. By Step~4, $\varphi(t)\Norm{e(t)} < \varepsilon < 1$, hence $\frac{1}{1 - \varphi(t)^2\Norm{e(t)}^2}$ is bounded. Moreover, since $\Vert \theta_1\Vert_{\infty} \leq \hat{\theta}_1$, it follows that $\xi_1(t) = \theta_1(t) + \frac{-e(t)}{1-\varphi(t)^2\Norm{e(t)}^2}$ is bounded. Further, since the difference $\dot{y} - \Gamma \xi_1$ is bounded by Step~2, we can conclude that $\dot{y}$ is bounded. Choose $b_0 \coloneq \sup_{t \in[t_0,\omega)} \Norm{\ddt \frac{-e(t)}{1-\varphi(t)^2\Norm{e(t)}^2}}$, which is bounded due to Step~4 and the boundedness of $\dot{y}$ yielding that $\dot{e}$ is bounded. Observe that this shows that $\dot{\theta}_1$ is bounded. Choose some $\frac{\Vert \theta_1(t_0) \Vert}{\hat{\theta}_1} < \tilde{\varepsilon}_1 < 1$ such that
$$ \hat{\theta}_1 ((r-1) \Vert \xi_1 \Vert_{\infty} + b_{0} + \hat{\theta}_{2}) < \frac{\tilde{\varepsilon}_1^2}{1-\tilde{\varepsilon}_1^2}.$$
Seeking a contradiction, assume that there exists $s_1\in [t_0,\omega)$ such that  $\Vert \theta_1(s_1) \Vert > \tilde{\varepsilon}_1 \hat{\theta}_1$ holds. Due to the condition on the initial condition, there exists $s_0 \coloneq \sup \{ s \in [t_0,s_1] \mid \Vert \theta_1(s) \Vert = \tilde{\varepsilon}_1 \hat{\theta}_1 \}$. For each $s \in [s_0,s_1]$, the estimate
\begin{align*}
& \tfrac{1}{2} \frac{d}{ds} \Norm{\theta_1(s)}^2 = \left\langle \theta_1(s), \dot{\xi}_1(s) - \frac{d}{ds} \frac{-e(s) }{1 - \varphi(s)^2\Norm{e(s)}^2} \right\rangle \\ 
&\qquad\leq \langle  \theta_1(s), -(r-1)\xi_1(s) + \xi_{2}(s) \rangle + \Norm{\theta_1} b_{0}\\
&\qquad\leq \Norm{\theta_1} ((r-1) \Vert \xi_1 \Vert_{\infty} + b_{0})\\
&\qquad\quad + \left\langle \theta_1(s), \theta_{2}(s) + \frac{-\theta_1(s)}{\hat{\theta}_1^2 - \Norm{\theta_1(s)}^2}\right\rangle \\
&\qquad\leq \hat\theta_1 ((r-1) \Vert \xi_1 \Vert_{\infty}\! +\! b_{0} \!+\! \hat{\theta}_{2}) + \frac{-\Norm{\theta_1(s)}^2}{\hat{\theta}_1^2 - \Norm{\theta_1(s)}^2}\\
&\qquad\leq \hat{\theta}_1 ((r-1) \Vert \xi_1 \Vert_{\infty} + b_{0} + \hat{\theta}_{2}) - \frac{\tilde{\varepsilon}_1^2}{1-\tilde{\varepsilon}_1^2} < 0
\end{align*}
holds. By integration, we obtain
$$ \tilde{\varepsilon}_1 \hat{\theta}_1< \theta_1(s_1) < \theta_1(s_0) = \tilde{\varepsilon}_1 \hat{\theta}_1 ,$$
a contradiction. We can conclude that $\Vert \theta_1(s) \Vert \leq \tilde{\varepsilon}_1 \hat{\theta}_1$ for all $s \in [t_0,\omega)$.

We continue to show the boundedness of $y^{(k)}, \frac{1}{\hat{\theta}_k^2 - \Norm{\theta_k}^2}, \dot{\theta}_{k-1}$ and $\xi_k$ inductively. Fix $2 \leq k \leq r-1$. Assume that all $\xi_{i}, y^{(i)}$, and $\frac{1}{\hat{\theta}_{i}^2 - \Norm{\theta_i}^2}$ for $1 \leq i < k$, and $\dot{\theta}_{i-1}$ for $2\leq i < k$ are bounded. We show that this implies that $\xi_{k}, \frac{1}{\hat{\theta}_k^2 - \Norm{\theta_k}^2}, \dot{\theta}_{k-1}$ and $y^{(k)} $ are bounded. 

Observe that $\xi_{k}(t) = \theta_{k}(t) + \frac{-\theta_{k-1}(t)}{\hat{\theta}_{k-1}^2 - \Norm{\theta_{k-1}(t)}^2}$ is bounded by the induction hypothesis. Further, we can use~\eqref{eq:Z_k} to conclude that
\begin{equation*}
y^{(k)}\!(t) \!=\! \frac{(\!-1\!)^{r-1-k}}{q_k}\! \!\left( \!Z_{k}(t)\! -\! \!\sum_{j=1}^{k} B_j^{(k)} \Gamma \xi_j(t)\! -\!\! \sum_{j=0}^{k-1} A_j^{(k)} y^{(j)}(t)\! \!\right)\!\! ,
\end{equation*}
which is bounded, since $Z_{k} $ is bounded by Step~3, and the functions $y,\ldots,y^{(k-1)}$ and $\xi_1,\ldots,\xi_{k-1}$ are bounded by the induction hypothesis, and $\xi_{k}$ remains bounded by the previous argument. Observe that $\dot{\theta}_{k-1}(t) = \dot{\xi}_{k-1}(t) - \ddt \frac{-\theta_{k-2}(t)}{\hat{\theta}^2_{k-2} - \Norm{\theta_{k-2}(t)}^2}$ is bounded, since $\frac{1}{\hat{\theta}^2_{k-2} - \Norm{\theta_{k-2}(t)}^2},\xi_{k}, \xi_{k-1}$ and $\dot{\theta}_{k-2}$ are bounded by the induction hypothesis. Define $b_{k-1} \coloneq \sup_{t \in [t_0,\omega)} \Norm{\ddt \frac{-\theta_{k-1}(t)}{\hat{\theta}_{k-1}^2 - \Norm{\theta_{k-1}(t)}^2}}$, which is finite by the induction hypothesis and the previously observed boundedness of $\dot{\theta}_{k-1}$. Choose some $\frac{\Vert \theta_k(t_0) \Vert}{\hat{\theta}_k} < \tilde{\varepsilon}_k < 1$ such that
$$ \hat{\theta}_k ((r-k) \Vert \xi_k \Vert_{\infty} + b_{k-1} + \hat{\theta}_{k+1}) < \frac{\tilde{\varepsilon}_k^2}{1-\tilde{\varepsilon}_k^2}.$$
Seeking a contradiction, assume that there exists $s_1\in [t_0,\omega)$ such that  $\Vert \theta_k(s_1) \Vert > \tilde{\varepsilon}_k \hat{\theta}_k$ holds. Due to the condition on the initial condition, there exists $s_0 \coloneq \sup \{ s \in [t_0,s_1] \mid \Vert \theta_k(s) \Vert = \tilde{\varepsilon}_k \hat{\theta}_k \}$. For each $s \in [s_0,s_1]$, the estimate
\begin{align*}
&\tfrac{1}{2} \frac{d}{ds} \Norm{\theta_k(s)}^2 = \left\langle \theta_k(s), \dot{\xi}_k(s) - \frac{d}{ds} \frac{-\theta_{k-1}(s)}{\hat{\theta}_{k-1}^2 - \Norm{\theta_{k-1}(s)}^2} \right\rangle \\ 
&\qquad\leq \langle  \theta_k(s), -(r-k)\xi_k(s) + \xi_{k+1}(s) \rangle + \Norm{\theta_k} b_{k-1}\\
&\qquad\leq \Norm{\theta_k} ((r-k) \Vert \xi_k \Vert_{\infty} + b_{k-1}) \\
&\qquad\quad+ \left\langle \theta_k(s), \theta_{k+1}(s) + \frac{-\theta_k(s)}{\hat{\theta}_k^2 - \Norm{\theta_k(s)}^2}\right\rangle \\
&\qquad\leq \hat\theta_k ((r-k) \Vert \xi_k \Vert_{\infty}\! + \!b_{k-1}\! +\! \hat{\theta}_{k+1}) \!+\! \frac{-\Norm{\theta_k(s)}^2}{\hat{\theta}_k^2 - \Norm{\theta_k(s)}^2}\\
&\qquad\leq \hat{\theta}_k ((r-k) \Vert \xi_k \Vert_{\infty} \!+\! b_{k-1} \!+\! \hat{\theta}_{k+1}) \!-\! \frac{\tilde{\varepsilon}_k^2}{1-\tilde{\varepsilon}_k^2}\! < \!0
\end{align*}
holds. By integration, we obtain
$$ \tilde{\varepsilon}_k \hat{\theta}_k< \theta_k(s_1) < \theta_k(s_0) = \tilde{\varepsilon}_k \hat{\theta}_k ,$$
a contradiction. We can conclude that $\Vert \theta_k(s) \Vert \leq \tilde{\varepsilon}_k \hat{\theta}_k$ for all $s \in [t_0,\omega)$, finishing the induction. 

It remains to show boundedness of $u$, which follows from the existence of $\tilde{\varepsilon}_{r-1} < 1$ satisfying $\Vert \theta_{r-1} \Vert \leq \tilde{\varepsilon}_{r-1} \hat{\theta}_{r-1}$, hence $\frac{1}{\hat{\theta}_{r-1}^2 - \Norm{\theta_{r-1}(t)}^2}$ is bounded, showing the claim of Step~5.

\emph{Step 6:} We show the existence of a global solution by utilizing the properties established in the previous steps. Assume that $\omega < \infty$. By Step~4 and the existence of $\tilde{\varepsilon}_1, \ldots, \tilde{\varepsilon}_{r-1} < 1$ as well as boundedness of $\xi_1,\ldots,\xi_{r-1}$ shown in Step~5, the closure of the graph of the maximal solution is a compact subset of $\cE$. This is contradicting the findings from Step~1. Therefore, $\omega = \infty$, showing the existence of a global solution, and thereby completing the proof of the theorem.
\end{proof}

\begin{remark} 
Note that the filter-based controller~\eqref{eq:Controller} uses only $r-1$ dynamic variables $\xi_1,\ldots,\xi_{r-1}$ and no further assumptions on the gain matrix $\Gamma$ are needed, beyond the positive definiteness of its symmetric part. In contrast, the funnel controller using a cascade of pre-compensators from~\cite{BergerReis18,Lanz22} introduces $r(r-1)$ dynamic variables and requires knowledge of an estimate close enough to~$\Gamma$.
\end{remark}

\begin{remark}
For the choice of the tuning parameters $\hat{\theta}_i$ and the initial values $\xi_i^0$ of the filter variables, note that their selection is coupled through the conditions on the initial conditions. If $y^0(t_0)$ is known, $\xi_i^0$ can be selected such that the initial estimation error $\theta_i^0$ is zero, thereby making the corresponding coupling condition trivial. If uncertainty remains regarding $y^0(t_0)$, the coupling conditions have to be satisfied for the admissible range of possible initial outputs, which will impose a nontrivial lower bound on the admissible values of $\hat{\theta}_i$.\\
Once these analytical constraints are satisfied, the remaining parameter selection is performed empirically. In our investigations, no universal relationship between an individual parameter and the resulting tracking performance or control effort was observed. Rather, acceptable performance is obtained over parameter ranges that depend on the values of the remaining parameters. This reflects the coupled influence of the design parameters on the closed-loop behavior and makes a simple sequential tuning rule difficult to establish. In practice, the parameters are therefore selected within their analytically admissible ranges and subsequently adjusted based on the desired tracking performance and control behavior.
\end{remark}

\begin{remark}
The proposed approach also involves a trade-off compared with conventional funnel control. Since the required output derivatives are estimated through additional filter dynamics, the resulting controller has a higher computational and dynamic complexity and introduces $r-1$ auxiliary filter states for a system of relative degree $r$. Furthermore, the filter parameters are coupled, which can make their selection less straightforward than in conventional funnel-control designs. These additional complexities constitute a trade-off for achieving derivative-free implementation while retaining the model-free nature of funnel control.
\end{remark}

\section{Simulations} \label{sec:Simulations}

To illustrate the main result of this work, we present a numerical simulation of the behavior of the closed-loop system~\eqref{eq:System},\eqref{eq:Controller}, conducted using MATLAB (solver: \texttt{ode15s}, rel.tol.: $10^{-8}$, abs.tol.: $10^{-6}$).

\begin{figure} 
\includegraphics[width = \columnwidth]{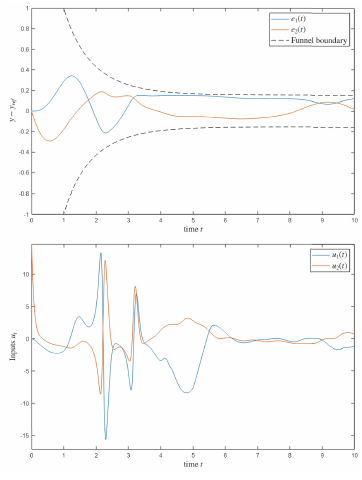}
\caption{Simulation of the closed-loop system~\eqref{NonlinSys},\eqref{eq:Controller}.}
\label{Fig:Nonlinear}
\end{figure}

In order to compare the performance with the controller proposed by~\cite{BergerReis18}, we consider the nonlinear system introduced in~\cite[Ex.~5.2]{Lanz22} that was examined with the control law introduced in~\cite{BergerReis18}, given by~\eqref{eq:System} with $r=3$, $t_0=0$ and
\begin{equation} \label{NonlinSys}
\begin{aligned}
R_1 &= \begin{bmatrix} 1 & -1 \\ 0 & 0 \end{bmatrix}, \quad
R_2 = \begin{bmatrix} 1 & 1 \\ 0 & -1 \end{bmatrix},\quad \Gamma = \begin{bmatrix} 2 & 0.2 \\ 0.2 & 2 \end{bmatrix},\\
d &: \R_{\geq 0} \rightarrow \R^2, t \mapsto \begin{pmatrix} 0.2\sin(5t)+0.2\cos(7t) \\ 0.25\sin(9t) + 0.2\cos(3t) \end{pmatrix}, \\
\oT &: \cC(\R_{\ge 0}, \R^2)^3 \rightarrow L_{\text{loc}}^{\infty}(\R_{\geq 0},\R^6),\\
& (\xi_1(\cdot),\xi_2(\cdot),\xi_3(\cdot)) \mapsto\\
&\quad \left(  t \mapsto  \begin{pmatrix} d_1(t)\\ d_2(t)\\ \xi_{1,1}(t)^2 + e^{\xi_{1,1}(t)-\vert \xi_{2,1}(t) \vert} \\ \xi_{1,2}(t)^3 - \sin(\xi_{2,2}(t)) \\ \int_0^t  e^{-(t-s)}\Norm{\xi_1(t)}^2 \tanh(\Norm{\xi_3(t)}^2) {\rm d}s  \\ \xi_{1,1}(t) \end{pmatrix}  \right), \\
\text{where } & \xi_i = (\xi_{i,1},\xi_{i,2})^{\top} \text{ for } 1 \leq i \leq 3, \text{ and }\\
f : \R^6 & \rightarrow \R^2, (z_1,\ldots,z_6) \! \mapsto \! \begin{pmatrix} z_1 + z_3 + z_5^3 - z_6 \\ z_2 + z_4 - z_5 \end{pmatrix}.
\end{aligned}
\end{equation}
We choose the reference trajectory $y_{\text{ref}}(t) = (e^{-(t-5)^2}, \sin(t))^{\top}$ and initial values are $y^0 = \dot{y}^0 = \ddot{y}^0 = 0$. Note that the external disturbance $d(t)$ can be incorporated into the model~\eqref{eq:System} via the operator $\oT$. 

The controller design parameters for~\eqref{eq:Controller} are chosen as $\vartheta = 1$, $\hat{\theta}_1 = 0.25$ and $\hat{\theta}_2 = 0.01$, and the filter variables are initialized by $\xi_1^0 = \xi_2^0 = 0$. The funnel function $\varphi(t) = (2.1(e^{-3t}+0.05) + 2e^{-t}+0.05)^{-1} $ is chosen such that it ensures the same transient behavior of the error $y-y_{\text{ref}}$ as in~\cite{Lanz22}. Observe that this yields $\Norm{\theta_1^0} = \Norm{\frac{-y_{\text{ref}}(0)}{1-\varphi(0)^2\Norm{y_{\text{ref}}(0)}^2}} < 10^{-10} < \hat{\theta}_1$ and $\Norm{\theta_2^0} = \Norm{\frac{\theta_1^0}{\hat{\theta}_1^2 - \Norm{\theta_1^0}^2}} < 10^{-8} < \hat{\theta}_2$. The simulation results are depicted in Figure~\ref{Fig:Nonlinear}. 

Compared to the performance of the funnel controller using a pre-compensator shown in~\cite{Lanz22}, the proposed controller exhibits a moderately higher bandwidth while maintaining comparable transient performance. Moreover, we note that the tuning process includes fewer parameters and the controller involves fewer auxiliary differential equations, in comparison to the funnel controller equipped with a cascade of pre-compensators.

\section{Conclusion} \label{sec:Conclusion}

In this work, we proposed a derivative-free extension of funnel control for nonlinear multi-input multi-output systems with arbitrary relative degree. By employing a set of filter variables as substitutes for the unavailable output derivatives, the resulting controller avoids the differentiation of the output signal while being of low complexity. The tracking error evolves within the prescribed performance funnel, and all signals remain bounded. Compared to earlier works based on the funnel pre-compensator, the new approach involves fewer dynamic variables and tuning parameters and avoids additional assumptions on the input gain matrix~$\Gamma$.

The proposed approach remains sensitive to measurement noise and discretization effects, as transient violations of the funnel constraint cannot be excluded. This is inherent to funnel control and not specific to the proposed extension. Future work may address the integration of the novel approach with other funnel-based control schemes, as well as the extension to more complex settings such as networked systems.

\bibliographystyle{elsarticle-harv}
\bibliography{lit}

\end{document}